\documentclass[reqno,b5paper]{amsart}
\usepackage{amsmath}
\usepackage{amssymb}
\usepackage{amsthm}
\usepackage{enumerate}
\usepackage[mathscr]{eucal}
\usepackage{eqlist}

\theoremstyle{plain}
\newtheorem{theorem}{Theorem}[section]

\newtheorem{lemma}{Lemma}[section]

\theoremstyle{definition}

\theoremstyle{remark}

\numberwithin{equation}{section}

\begin{document}
\title[ON THE DIOPHANTINE EQUATION $x^{2}+2^{a}\cdot 3^{b}\cdot
11^{c}=\allowbreak y^{n}$]
{ON THE DIOPHANTINE EQUATION $x^{2}+2^{a}\cdot 3^{b}\cdot
11^{c}=\allowbreak y^{n}$}
\author[I. N. CANGUL, M. DEMIRCI, I. INAM, \and F. LUCA AND
G. SOYDAN]%
{ISMAIL NACI CANGUL*, MUSA DEMIRCI*, ILKER INAM*, \and FLORIAN LUCA** AND
G\"{O}KHAN SOYDAN***}

\newcommand{\acr}{\newline\indent}

\address{\llap{*\,}Department of Mathematics\acr 
                    Uluda\u{g} University\acr 
                    16059 Bursa\acr
                    TURKEY}
\email{cangul@uludag.edu.tr, mdemirci@uludag.edu.tr, inam@uludag.edu.tr}

\address{\llap{**\,}Mathematical Institute\acr 
                    Universidad Nacional Aut\'{o}noma de M\'{e}xico\acr 
                    CP 58089, Morelia, Michoac\'{a}n\acr
                    MEXICO}
\email{fluca@matmor.unam.mx}

\address{\llap{***\,}I\c{s}{\i}klar Air Force High School\acr
                   16039 Bursa\acr
                   TURKEY}
\email{gsoydan@uludag.edu.tr}

\thanks{The first author is supported by Turkish Research Council Project No: 107T311 and by Uludag University Research Fund Project Numbers 2006$-$40, 2008$-$31 and 2008$-$54.\linebreak 
\indent The third author is supported by Turkish Research Council Project No: 107T311.}


\begin{abstract}
In this note, we find all the solutions of the Diophantine equation $%
x^{2}+2^{a}\cdot 3^{b}\cdot 11^{c}=\allowbreak y^{n},$ in nonnegative
integers $a,~b,~c,~x,~y,~n\geq 3$ with $x$ and $y$ coprime.
\end{abstract}

\maketitle

\section{Introduction}

The history of the Diophantine equation 
\begin{equation}
x^{2}+C=y^{n},\qquad n\geq 3  \label{eq:(1.1)}
\end{equation}%
in positive integers $x$ and $y$ goes back to 1850's. In 1850, Lebesque \cite%
{Lebesque} proved that the equation \eqref{eq:(1.1)} has no solutions when $%
C=1$. This equation is a particular case of the Diophantine equation $%
ay^{2}+by+c=dx^{n},$ where $a\neq 0,~b,~c$ and $d\neq 0$ are integers with $%
b^{2}-4ac\neq 0.$ This equation has at most finitely many integer solutions $%
x,~y,~n\geq 3$. This was proved to be so by Landau and Ostrowski (see \cite%
{Landau}) for a fixed $n\geq 3$ in 1920. The fact that $n$ itself is also
bounded was proved only 63 years later by Stewart and Shorey in \cite{SS}.
For the best theoretical upper bounds available today on the exponent $n$,
we refer to \cite{Berczes2} and \cite{Gyry}. However, these estimates are
based on Baker's theory of lower bounds for linear forms in logarithms of
algebraic numbers, so they are quite impractical.

We next survey some results concerning the actual resolution of the
Diophantine equation \eqref{eq:(1.1)} for various values of $C$. In 1993,
Cohn \cite{Cohn1} studied the Diophantine equation \eqref{eq:(1.1)} and
found all its integer solutions $(x,y,n)$ for most values of $C$ in the
interval $[1,100].$ In \cite{Mignotte}, Mignotte and de Weger dealt with the
cases $C=74$ and $86$, which had not been covered by Cohn. In both these
cases, the only interesting value of the exponent $n$ is $n=5$. The
remaining cases were finally dealt with by Bugeaud, Mignotte and Siksek in 
\cite{Bugeaud2}.

Variations of the Diophantine equation \eqref{eq:(1.1)} have also been
intensively studied. For example, if we replace $y^n$ from the right hand
side of equation \eqref{eq:(1.1)} with $2y^n$, but keep the conditions that $%
n\ge 3$ and the coprimality condition on $x$ and $y$, we get an equation
whose solutions were found \cite{Tengely} for all values of $C$ which are
squares of an odd integer $B\in \{3,5,7,\ldots,501\}$.

Recently, several authors studied the case when $C$ is a positive integer
which is an ${\mathcal{S}}$-unit, where ${\mathcal{S}}$ is some small set of
primes. Recall that if ${\mathcal{S}}:=\{p_{1},\ldots ,p_{k}\}$ is some
finite set of primes, then an ${\mathcal{S}}$-unit is an integer all whose
prime factors are in ${\mathcal{S}}$. When ${\mathcal{S}}=\{2\}$, we obtain
the equation $x^{2}+2^{k}=y^{n}$ which was first studied by Cohn in 1992
(see \cite{Cohn2}) who found all the integer solutions $(x,y,k,n)$ with $%
n\geq 3$ when $k$ is odd, even without the coprimality condition on $x$ and $%
y$. The case when $k$ is even in the above Diophantine equation generated a
few papers before it was finally completely settled in \cite{Abu2001},
without the coprimality condition on $x$ and $y$, and independently but one
year later by Le \cite{Le1}, under the coprimality condition on $x$ and $y$.
The two proofs used different tools. The recorded work on the Diophantine
equation $x^{2}+3^{k}=y^{n}$ is more entertaining. In 1998, Abu Muriefah and
Arif \cite{Arif1}, found its solutions with $k$ odd and two years later, in
2000, Luca \cite{Luca3}, found all the solutions with $k$ even. Unaware of
this work, the above results were rediscovered in 2008 by Tao Liqun \cite%
{Liqun}. The case when ${\mathcal{S}}=\{5\}$ was dealt with in \cite{Arif2}
and \cite{FSAbu7}. Partial results on the case when  ${\mathcal{S}}=\{7\}$ appear in 
\cite{Luca4}. Recently, B\'{e}rczes and Pink \cite{Berczes}, found all the
solutions of the Diophantine equation \eqref{eq:(1.1)} when $C=p^{k}$ and $k$
is even, where $p$ is any prime in the interval $[2,100]$.

The paper \cite{Luca1} is the first recorded instance in which all the
solutions of the Diophantine equation \eqref{eq:(1.1)} were found when $C$
is some positive ${\mathcal{S}}$-unit for a set ${\mathcal{S}}$ containing
more than one prime. In that instance, the set ${\mathcal{S}}$ was $\{2,3\}$%
. Since then, all solutions of the same Diophantine equation \eqref{eq:(1.1)}
when $C>0$ is some ${\mathcal{S}}$-unit were found in \cite{Luca2} for ${%
\mathcal{S}}=\{2,5\}$, in \cite{FSAbu3} for ${\mathcal{S}}=\{5,13\}$, in 
\cite{CDLPS} for ${\mathcal{S}}=\{2,11\}$, and in \cite{Luca5} for ${%
\mathcal{S}}=\{2,5,13\}$. In \cite{Pink} Pink has obtained some results for $%
{\mathcal{S}}=\{2,3,5,7\}$.

Here, we add to the literature on the topic and study the case when $C>0$ is
an ${\mathcal{S}}$-unit, where ${\mathcal{S}}=\{2,3,11\}$. 
More precisely, we study the Diophantine equation 
\begin{equation}  \label{eq:(1.2)}
x^{2}+2^{a}\cdot 3^{b}\cdot 11^{c}=y^{n},\qquad (x,y)=1\qquad {\text{\textrm{%
and}}}\qquad n\geq 3.
\end{equation}

Our result is the following.

\begin{theorem}
The only solutions of the Diophantine equation \eqref{eq:(1.2)} are: 
\begin{eqnarray*}
n &=&3:\text{ \ \ \ \ \ \ \ \ \ \ \ \ the solutions given in Table 1 and
Table 2;} \\
n &=&4:\text{ \ \ \ \ \ \ \ \ \ \ \ \ the solutions given in Table 3;} \\
n &=&5:\text{ \ \ \ \ \ \ \ \ \ \ \ \ }%
(x,y,a,b,c)=(1,3,1,0,2),~(241,9,3,0,2); \\
n &=&6:~\ \ \ \ \ \ \ \ \ \ \ (x,y,a,b,c)\in
\{(5,3,6,0,1),(37,5,4,4,1),~(117,5,4,0,2)\}; \\
n &=&10:~\ \ \ \ \ \ \ \ \ \ (x,y,a,b,c)=(241,3,3,0,2);
\end{eqnarray*}
\end{theorem}

A few words about the proofs.

We start by treating the cases $n=3$ and $n=4$. This is achieved in Section
2 and Section 3, respectively. As a method, we transform equation %
\eqref{eq:(1.2)} into several elliptic equations written in cubic and
quartic models, respectively, for which we need to determine all their $%
\{2,3,11\}$-integral points. As a byproduct of our results, we also read
easily that the only exponents $n\geq 3$ whose prime factors are in the set $%
\{2,3\}$ and for which equation \eqref{eq:(1.2)} has a solution $%
(x,y,a,b,c,n)$ are $n=3,4,6$. In Section 4, we assume that $n\geq 5$ and
study the equation \eqref{eq:(1.2)} under this assumption. The method here
uses the properties of the Primitive Divisors of Lucas sequences. All the
computations are done with MAGMA \cite{Bosma} and with Cremona's program
mwrank.

Before digging into the proofs, we note that since $n\geq 3$, it follows
that $n$ is either a multiple of $4$, or $n$ is a multiple of an odd prime $%
p.$ Furthermore, if $d~|~n$ is such that $d\in \{4,p\}$ with $p$ an odd
prime and $(x,y,a,b,c,n)$ is a solution of our equation \eqref{eq:(1.2)},
then $(x,y^{n/d},a,b,c,d)$ is also a solution of our equation %
\eqref{eq:(1.2)} satisfying the same restrictions. Thus, we may replace $n$
by $d$ and $y$ by $y^{n/d}$, and from now on assume that $n\in \{4,p\}.$
Furthermore, note that when $c=0$ our equation becomes $%
x^{2}+2^{a}3^{b}=y^{n}$ all solutions of which are already known from what
we have said earlier (see \cite{Luca1}), while when $b=0$ our equation
becomes $x^{2}+2^{a}11^{c}=y^{n}$ all solutions of which have been found in 
\cite{CDLPS}. Thus, we shall assume that $bc>0$. Since $3^{b}11^{c}\equiv
1,~3\pmod 8$ according to whether $b+c$ is even or odd, it follows by
considerations modulo $8$ that either $a>0$, or that $x$ is even. This
observation will be useful later on.
\newpage
\begin{table}[ht]
\caption{Solutions for $n=3$.}
\begin{center}
\begin{tabular}{|l|l|l|l|l|l|l|l|l|}
\hline
$\alpha $ & $\beta $ & $\gamma $ & $z$ & $a$ & $b$ & $c$ & $x$ & $y$ \\ \hline
$0$ & $0$ & $1$ & $1$ & $0$ & $0$ & $1$ & $4$ & $3$ \\ 
$0$ & $0$ & $1$ & $1$ & $0$ & $0$ & $1$ & $58$ & $15$ \\ 
$0$ & $0$ & $1$ & $2$ & $0$ & $6$ & $1$ & $5066$ & $295$ \\ 
$0$ & $0$ & $1$ & $2$ & $6$ & $0$ & $1$ & $5$ & $9$ \\ 
$0$ & $0$ & $1$ & $8$ & $18$ & $0$ & $1$ & $6179$ & $345$ \\ 
$0$ & $0$ & $1$ & $6$ & $6$ & $6$ & $1$ & $5491$ & $313$ \\ 
$0$ & $0$ & $2$ & $1$ & $0$ & $0$ & $2$ & $2$ & $5$ \\ 
$0$ & $0$ & $2$ & $2$ & $6$ & $0$ & $2$ & $835$ & $89$ \\ 
$0$ & $0$ & $2$ & $4$ & $12$ & $0$ & $2$ & $404003$ & $5465$ \\ 
$0$ & $0$ & $2$ & $3$ & $0$ & $6$ & $2$ & $908$ & $97$ \\ 
$0$ & $0$ & $3$ & $1$ & $0$ & $0$ & $3$ & $9324$ & $443$ \\ 
$0$ & $0$ & $4$ & $6$ & $6$ & $6$ & $4$ & $589229$ & $7033$ \\ 
$0$ & $1$ & $2$ & $3$ & $0$ & $7$ & $2$ & $910$ & $103$ \\ 
$0$ & $1$ & $4$ & $9$ & $0$ & $13$ & $4$ & $14259970$ & $58807$ \\ 
$0$ & $2$ & $2$ & $33$ & $0$ & $8$ & $8$ & $1549034$ & $15613$ \\ 
$0$ & $3$ & $2$ & $6$ & $6$ & $9$ & $2$ & $4085$ & $553$ \\ 
$0$ & $4$ & $0$ & $1$ & $0$ & $4$ & $0$ & $46$ & $13$ \\ 
$0$ & $4$ & $1$ & $1$ & $0$ & $4$ & $1$ & $170$ & $31$ \\ 
$0$ & $4$ & $3$ & $4$ & $12$ & $4$ & $3$ & $239363$ & $3865$ \\ 
$0$ & $5$ & $0$ & $1$ & $0$ & $5$ & $0$ & $10$ & $7$ \\ 
$0$ & $5$ & $1$ & $1$ & $0$ & $5$ & $1$ & $7910$ & $397$ \\ 
\hline
\end{tabular}
\end{center}
\end{table}
\newpage
\begin{table}[ht]
\caption{Solutions for $n=3$.}
\begin{center}
\begin{tabular}{|l|l|l|l|l|l|l|l|l|}
\hline
$\alpha $ & $\beta $ & $\gamma $ & $z$ & $a$ & $b$ & $c$ & $x$ & $y$ \\ \hline
$1$ & $0$ & $0$ & $1$ & $1$ & $0$ & $0$ & $5$ & $3$ \\ 
$1$ & $0$ & $2$ & $1$ & $1$ & $0$ & $2$ & $5805$ & $323$ \\ 
$1$ & $3$ & $0$ & $1$ & $1$ & $3$ & $0$ & $17$ & $7$ \\ 
$1$ & $3$ & $1$ & $3$ & $1$ & $9$ & $1$ & $1043$ & $115$ \\ 
$1$ & $4$ & $1$ & $2$ & $7$ & $4$ & $1$ & $2196415$ & $16897$ \\ 
$1$ & $5$ & $1$ & $1$ & $1$ & $5$ & $1$ & $865$ & $91$ \\ 
$1$ & $5$ & $1$ & $9$ & $1$ & $17$ & $1$ & $94517$ & $2275$ \\ 
$2$ & $0$ & $0$ & $1$ & $2$ & $0$ & $0$ & $11$ & $5$ \\ 
$2$ & $0$ & $0$ & $22$ & $8$ & $0$ & $6$ & $5497$ & $785$ \\ 
$2$ & $0$ & $1$ & $1$ & $2$ & $0$ & $1$ & $9$ & $5$ \\ 
$2$ & $0$ & $1$ & $6$ & $8$ & $6$ & $1$ & $4069$ & $265$ \\ 
$2$ & $2$ & $2$ & $3$ & $2$ & $8$ & $2$ & $241397$ & $3877$ \\ 
$2$ & $4$ & $1$ & $1$ & $2$ & $4$ & $1$ & $217$ & $37$ \\ 
$2$ & $4$ & $2$ & $1$ & $2$ & $4$ & $2$ & $107$ & $37$ \\ 
$2$ & $4$ & $5$ & $2$ & $8$ & $4$ & $5$ & $335802455$ & $483121$ \\ 
$2$ & $5$ & $0$ & $1$ & $2$ & $5$ & $0$ & $35$ & $13$ \\ 
$2$ & $5$ & $2$ & $1$ & $2$ & $5$ & $2$ & $1495$ & $133$ \\ 
$2$ & $5$ & $2$ & $2$ & $8$ & $5$ & $2$ & $34255$ & $1057$ \\ 
$3$ & $2$ & $2$ & $3$ & $3$ & $8$ & $2$ & $912668635$ & $940897$ \\ 
$3$ & $4$ & $0$ & $1$ & $3$ & $4$ & $0$ & $955$ & $97$ \\ 
$3$ & $5$ & $1$ & $3$ & $3$ & $11$ & $1$ & $8099$ & $433$ \\ 
$4$ & $0$ & $2$ & $1$ & $4$ & $0$ & $2$ & $117$ & $25$ \\ 
$4$ & $0$ & $3$ & $2$ & $10$ & $0$ & $3$ & $9959$ & $465$ \\ 
$4$ & $1$ & $0$ & $3$ & $4$ & $7$ & $0$ & $595$ & $73$ \\ 
$4$ & $2$ & $1$ & $12$ & $16$ & $8$ & $1$ & $73225$ & $2161$ \\ 
$4$ & $4$ & $0$ & $1$ & $4$ & $4$ & $0$ & $2681$ & $193$ \\ 
$4$ & $4$ & $1$ & $1$ & $4$ & $4$ & $1$ & $37$ & $25$ \\ 
$4$ & $4$ & $1$ & $1$ & $4$ & $4$ & $1$ & $97129$ & $2113$ \\ 
$4$ & $4$ & $1$ & $2$ & $10$ & $4$ & $1$ & $17$ & $97$ \\ 
$4$ & $4$ & $4$ & $1$ & $4$ & $4$ & $4$ & $3419$ & $313$ \\ 
$4$ & $5$ & $2$ & $1$ & $4$ & $5$ & $2$ & $665$ & $97$ \\ 
$5$ & $1$ & $1$ & $3$ & $5$ & $7$ & $1$ & $53333$ & $1417$ \\ 
$5$ & $5$ & $0$ & $1$ & $5$ & $5$ & $0$ & $39151$ & $1153$ \\ \hline
\end{tabular}
\end{center}
\end{table}
\newpage
\begin{table}[hb]
\caption{Solutions for $n=4$.}
\begin{center}
\begin{tabular}{|l|l|l|l|l|l|l|l|l|}
\hline
$\alpha $ & $\beta $ & $\gamma $ & $z$ & $a$ & $b$ & $c$ & $x$ & $y$ \\ \hline  
$0$ & $3$ & $1$ & $4$ & $8$ & $3$ & $1$ & $233$ & $19$ \\ 
$0$ & $3$ & $2$ & $4$ & $8$ & $3$ & $2$ & $1607$ & $43$ \\ 
$1$ & $0$ & $0$ & $2$ & $5$ & $0$ & $0$ & $7$ & $3$ \\ 
$1$ & $0$ & $1$ & $6$ & $5$ & $4$ & $1$ & $7$ & $13$ \\ 
$1$ & $1$ & $0$ & $2$ & $5$ & $1$ & $0$ & $23$ & $5$ \\ 
$1$ & $1$ & $1$ & $4$ & $9$ & $1$ & $1$ & $4223$ & $65$ \\ 
$1$ & $1$ & $2$ & $6$ & $5$ & $5$ & $2$ & $235223$ & $485$ \\ 
$2$ & $1$ & $0$ & $2$ & $6$ & $1$ & $0$ & $47$ & $7$ \\ 
$2$ & $1$ & $1$ & $2$ & $6$ & $1$ & $1$ & $17$ & $7$ \\ 
$2$ & $1$ & $1$ & $2$ & $6$ & $1$ & $1$ & $527$ & $23$ \\ 
$2$ & $1$ & $1$ & $4$ & $10$ & $1$ & $1$ & $223$ & $17$ \\ 
$2$ & $1$ & $2$ & $2$ & $6$ & $1$ & $2$ & $73$ & $13$ \\ 
$2$ & $2$ & $0$ & $2$ & $6$ & $2$ & $0$ & $7$ & $5$ \\ 
$3$ & $1$ & $1$ & $1$ & $3$ & $1$ & $1$ & $19$ & $5$ \\ 
$3$ & $2$ & $0$ & $2$ & $7$ & $2$ & $0$ & $287$ & $17$ \\ 
$3$ & $2$ & $1$ & $2$ & $7$ & $2$ & $1$ & $343$ & $19$ \\ 
$3$ & $3$ & $1$ & $1$ & $3$ & $3$ & $1$ & $5$ & $7$ \\ 
$3$ & $3$ & $1$ & $3$ & $3$ & $7$ & $1$ & $2165$ & $47$ \\ \hline
\end{tabular}
\end{center}
\end{table}
\newpage
\section{The case when $n=3$}

\begin{lemma}
All solutions with $n=3$ and $bc>0$ of the Diophantine equation $%
\eqref{eq:(1.2)}$ are given in Tables $4$ and $5$:
\end{lemma}

\begin{table}[ht]
\caption{Solutions for $n=3$.}
\begin{center}
\begin{tabular}{|l|l|l|l|l|l|l|l|l|}
\hline
$\alpha $ & $\beta $ & $\gamma $ & $z$ & $a$ & $b$ & $c$ & $x$ & $y$ \\ \hline 
$0$ & $0$ & $1$ & $2$ & $0$ & $6$ & $1$ & $5066$ & $295$ \\ 
$0$ & $0$ & $1$ & $6$ & $6$ & $6$ & $1$ & $5491$ & $313$ \\ 
$0$ & $0$ & $2$ & $3$ & $0$ & $6$ & $2$ & $908$ & $97$ \\ 
$0$ & $0$ & $4$ & $6$ & $6$ & $6$ & $4$ & $589229$ & $7033$ \\ 
$0$ & $1$ & $2$ & $3$ & $0$ & $7$ & $2$ & $910$ & $103$ \\ 
$0$ & $1$ & $4$ & $9$ & $0$ & $13$ & $4$ & $14259970$ & $58807$ \\ 
$0$ & $2$ & $2$ & $33$ & $0$ & $8$ & $8$ & $1549034$ & $15613$ \\ 
$0$ & $3$ & $2$ & $6$ & $6$ & $9$ & $2$ & $4085$ & $553$ \\ 
$0$ & $4$ & $3$ & $4$ & $12$ & $4$ & $3$ & $239363$ & $3865$ \\ 
$0$ & $4$ & $1$ & $1$ & $0$ & $4$ & $1$ & $170$ & $31$ \\ 
$0$ & $5$ & $1$ & $1$ & $0$ & $5$ & $1$ & $7910$ & $397$ \\ 
$1$ & $3$ & $1$ & $3$ & $1$ & $9$ & $1$ & $1043$ & $115$ \\ 
$1$ & $4$ & $1$ & $2$ & $7$ & $4$ & $1$ & $2196415$ & $16897$ \\ 
$1$ & $5$ & $1$ & $1$ & $1$ & $5$ & $1$ & $865$ & $91$ \\ 
$1$ & $5$ & $1$ & $9$ & $1$ & $17$ & $1$ & $94517$ & $2275$ \\ \hline
\end{tabular}
\end{center}
\end{table}

\begin{table}[ht]
\caption{Solutions for $n=3$.}
\begin{center}
\begin{tabular}{|l|l|l|l|l|l|l|l|l|}
\hline
$\alpha $ & $\beta $ & $\gamma $ & $z$ & $a$ & $b$ & $c$ & $x$ & $y$ \\ \hline
$2$ & $0$ & $1$ & $6$ & $8$ & $6$ & $1$ & $4069$ & $265$ \\ 
$2$ & $2$ & $2$ & $3$ & $2$ & $8$ & $2$ & $241397$ & $3877$ \\ 
$2$ & $4$ & $1$ & $1$ & $2$ & $4$ & $1$ & $217$ & $37$ \\ 
$2$ & $4$ & $2$ & $1$ & $2$ & $4$ & $2$ & $107$ & $37$ \\ 
$2$ & $4$ & $5$ & $2$ & $8$ & $4$ & $5$ & $335802455$ & $483121$ \\ 
$2$ & $5$ & $2$ & $1$ & $2$ & $5$ & $2$ & $1495$ & $133$ \\ 
$2$ & $5$ & $2$ & $2$ & $8$ & $5$ & $2$ & $34255$ & $1057$ \\ 
$3$ & $2$ & $2$ & $3$ & $3$ & $8$ & $2$ & $912668635$ & $940897$ \\ 
$3$ & $5$ & $1$ & $3$ & $3$ & $11$ & $1$ & $8099$ & $433$ \\ 
$4$ & $2$ & $1$ & $12$ & $16$ & $8$ & $1$ & $73225$ & $2161$ \\ 
$4$ & $4$ & $1$ & $1$ & $4$ & $4$ & $1$ & $37$ & $25$ \\ 
$4$ & $4$ & $1$ & $1$ & $4$ & $4$ & $1$ & $97129$ & $2113$ \\ 
$4$ & $4$ & $1$ & $2$ & $10$ & $4$ & $1$ & $17$ & $97$ \\ 
$4$ & $4$ & $4$ & $1$ & $4$ & $4$ & $4$ & $3419$ & $313$ \\ 
$4$ & $5$ & $2$ & $1$ & $4$ & $5$ & $2$ & $665$ & $97$ \\ 
$5$ & $1$ & $1$ & $3$ & $5$ & $7$ & $1$ & $53333$ & $1417$ \\ \hline
\end{tabular}
\end{center}
\end{table}
\newpage
\textit{In particular, if $n\geq 3$ is a multiple of $3$ and the Diophantine
equation \eqref{eq:(1.2)} has an integer solution $(x,y,a,b,c,n)$, then $n=6$%
. Furthermore, when $n=6$, the only solution $(x,y,a,b,c)$ is $(37,5,4,4,1).$%
}

\begin{proof}
Equation \eqref{eq:(1.2)} can be rewritten as 
\begin{equation}
\left( \frac{x}{z^{3}}\right) ^{2}+A=\left( \frac{y}{z^{2}}\right) ^{3},
\label{eq:2.1}
\end{equation}%
where $A$ is cubefree and defined implicitly by the relation $2^{a}\cdot
3^{b}\cdot 11^{c}=Az^{6}.$ One can see that $A=2^{\alpha }\cdot 3^{\beta
}\cdot 11^{\gamma }$ with some exponents $\alpha ,~\beta ,~\gamma \in
\left\{ 0,1,2,3,4,5\right\} .$ We thus get 
\begin{equation}
V^{2}=U^{3}-2^{\alpha }\cdot 3^{\beta }\cdot 11^{\gamma },  \label{eq:2.2}
\end{equation}%
where $U=y/z^{2}$, $V=x/z^{3}$, $\alpha ,~\beta ,~\gamma \in \left\{
0,1,2,3,4,5\right\} $ and all prime factors of $z$ are in $\{2,3,11\}$.
Thus, we need to determine all the $\left\{ 2,3,11\right\} $-integral points
on the totality of the 216 elliptic curves above. Recall that if ${\mathcal{S%
}}$ is a finite set of prime numbers, then an ${\mathcal{S}}$-integer is
rational number $a/b$ where $a$ and $b$ are coprime integers and $b$ is an ${%
\mathcal{S}}$-unit. We use MAGMA \cite{Bosma} to determine all the $\left\{
2,3,11\right\} $-integral points on the above elliptic curves from which one
can reconstruct easily the solutions $(x,y,a,b,c)$ listed in Tables 4 and
Table 5.

When $n=6$, we replace $n$ by $3$ and $y$ by $y^{2}$ and get a solution of
the same equation \eqref{eq:(1.2)} with $n=3$ and the value of $y$ being a perfect square. 
Looking in Tables $4$ and $5$ we get only the possibility $%
(37,25,4,4,1)$ for $(x,y,a,b,c)$.
Therefore, the only solution to equation %
\eqref{eq:(1.2)} having $n=6$ is $(37,5,4,4,1)$. 
This completes the proof of this lemma.
\end{proof}

\section{ The case $n=4$}

\begin{lemma}
The only solutions with $n=4$ and $bc>0$ of the Diophantine equation %
\eqref{eq:(1.2)} are given Table $6$.
\end{lemma}

\begin{table}[ht]
\caption{Solutions for $n=4$.}
\begin{center}
\begin{tabular}{|l|l|l|l|l|l|l|l|l|}
\hline
$\alpha $ & $\beta $ & $\gamma $ & $z$ & $a$ & $b$ & $c$ & $x$ & $y$ \\ \hline  
$0$ & $3$ & $1$ & $4$ & $8$ & $3$ & $1$ & $233$ & $19$ \\ 
$0$ & $3$ & $2$ & $4$ & $8$ & $3$ & $2$ & $1607$ & $43$ \\ 
$1$ & $0$ & $1$ & $6$ & $5$ & $4$ & $1$ & $7$ & $13$ \\ 
$1$ & $1$ & $1$ & $4$ & $9$ & $1$ & $1$ & $4223$ & $65$ \\ 
$1$ & $1$ & $2$ & $6$ & $5$ & $5$ & $2$ & $235223$ & $485$ \\ 
$2$ & $1$ & $1$ & $2$ & $6$ & $1$ & $1$ & $17$ & $7$ \\ 
$2$ & $1$ & $1$ & $2$ & $6$ & $1$ & $1$ & $527$ & $23$ \\ 
$2$ & $1$ & $1$ & $4$ & $10$ & $1$ & $1$ & $223$ & $17$ \\ 
$2$ & $1$ & $2$ & $2$ & $6$ & $1$ & $2$ & $73$ & $13$ \\ 
$3$ & $1$ & $1$ & $1$ & $3$ & $1$ & $1$ & $19$ & $5$ \\ 
$3$ & $2$ & $1$ & $2$ & $7$ & $2$ & $1$ & $343$ & $19$ \\ 
$3$ & $3$ & $1$ & $1$ & $3$ & $3$ & $1$ & $5$ & $7$ \\ 
$3$ & $3$ & $1$ & $3$ & $3$ & $7$ & $1$ & $2165$ & $47$ \\ 
\hline
\end{tabular}
\end{center}
\end{table}

\begin{proof}
We use a similar method as in the case $n=3$ except that we write equation %
\eqref{eq:(1.2)} as 
\begin{equation}
U^{2}+A=V^{4},  \label{3.1}
\end{equation}%
where now $U=x/z^{2},~V=y/z$, and $A$ is fourth powerfree and defined
implicitly by the relation $2^{a}\cdot 3^{b}\cdot 11^{c}=Az^{4}.$ Thus, $%
A=2^{\alpha }\cdot 3^{\beta }\cdot 11^{\gamma }$ holds with some exponents $%
\alpha ,~\beta ,\gamma \in \left\{ 0,1,2,3\right\} .$ Observe that all prime
factors of $z$ are in the set $\{2,3,11\}$. Hence, we reduced the problem to
determining all the $\{2,3,11\}$-integral points the totality of the 64
elliptic curves above. We used again MAGMA to determine these points from
which we easily determined all the corresponding solutions $(x,y,a,b,c)$
listed in Table $6$.
\end{proof}

\section{The case when $n\geq 5$ is prime}

\begin{lemma}
The Diophantine equation \eqref{eq:(1.2)} has no solutions with $n\ge 5$
prime and $bc>0$.
\end{lemma}

\begin{proof}
We change $n$ to $p$ to emphasize that $p~$is a prime number. We rewrite the
Diophantine equation \eqref{eq:(1.2)} as $x^2+dz^2=y^p,$ where 
\begin{equation}  \label{d}
d\in \{1,~2,~3,~6,~11,~22,~33,~66\}
\end{equation}
according to the parities of the exponents $a,~b$ and $c.$ Here, $%
z=2^{\alpha _{1}}\cdot 3^{\beta _{1}}\cdot 11^{\gamma _{1}}$ for some
nonnegative exponents $\alpha _{1}$, $\beta _{1}$ and$~\gamma _{1}$. Write ${%
\mathbb{K}}:=\mathbb{Q}(i\sqrt{d}).$ Observe that since $bc>0$ and either $%
a>0$ or $x$ is even (see the end of Section 1), it follows that $y$ is
always odd. A standard argument tells us now that in ${\mathbb{K}}$ we have 
\begin{equation}
(x+i\sqrt{d}z)(x-i\sqrt{d}z)=y^{n},
\end{equation}
where the ideals generated by $x+iz\sqrt{d}$ and $x-iz\sqrt{d}$ are coprime
in ${\mathbb{K}}$. Hence, the ideal $x+iz\sqrt{d}$ is a $p$th power of some
ideal ${\mathcal{O}}_{{\mathbb{K}}}$. The class number of ${\mathbb{K}}$
belongs to $\{1,2,4,8\}$. In particular, it is coprime to $p$. Thus, by a
standart argument, it follows that $x+iz\sqrt{d}$ is associated to a $p^{th}$
power in ${\mathcal{O}}_{{\mathbb{K}}}.$ The cardinality of the group of
units of ${\mathcal{O}}_{{\mathbb{K}}}$ is $2,~4$, or $6$, all coprime to $p$%
. Furthermore, $\{1,i\sqrt{d}\}$ is always an integral base for ${\mathcal{O}%
}_{{\mathbb{K}}}$ except for when $d=3$, and $d=11$, in which cases an
integral basis for ${\mathcal{O}}_{{\mathbb{K}}}$ is $\{1,(1+i \sqrt{d})/{2}%
\}$. Thus, we may assume that the relation 
\begin{equation}  \label{4.2}
x+i\sqrt{d}z=\eta ^{p}
\end{equation}%
holds with some algebraic integer $\eta \in {\mathcal{O}}_{{\mathbb{K}}}.$
We write $\eta =u+i\sqrt{d}v$, where either both $u$ and $v$ are integers,
or both $2u$ and $2v$ are odd integers, the last case occurring only when $%
d=3$ or $d=11.$ Conjugating equation \eqref{4.2} and subtracting the two
relations, we get 
\begin{equation}  \label{4.3}
2i\sqrt{d}\cdot 2^{\alpha _{1}}\cdot 3^{\beta _{1}}\cdot 11^{\gamma
_{1}}=\eta ^{p}-{\overline{\eta }}^{p}.
\end{equation}
The right hand side of the above equation is an integer multiple of $2i\sqrt{%
d}v=\eta-{\overline{\eta}}.$ We deduce that $v\mid 2^{\alpha _{1}}\cdot
3^{\beta _{1}}\cdot 11^{\gamma _{1}}$, and that 
\begin{equation}  \label{4.4}
\frac{2^{\alpha _{1}}\cdot 3^{\beta _{1}}\cdot 11^{\gamma _{1}}}{v}=\frac{
\eta ^{p}-{\overline{\eta} }^{p}}{\eta -{\overline{\eta }}}\in \mathbb{Z}.
\end{equation}
Now let $\{L_{m}\}_{m\geq 0}$ be the sequence of general term $L_{m}=(\eta
^{m}- {\overline{\eta }}^{m})/(\eta -{\overline{\eta }})$ for all $m\geq 0.$
This is a \textit{Lucas sequence} and it consists of integers. Its
discriminant is $(\eta -{\overline{\eta} })^{2}=-4dv^2.$ For nonzero integer 
$k$, let $P(k)$ be the largest prime factor of $k$ with the convention that $%
P(\pm 1)=1$. Equation \eqref{4.4} now leads to the conclusion that 
\begin{equation}  \label{4.5}
P(L_{p})=P\left(\frac{2^{\alpha _{1}}\cdot 3^{\beta _{1}}\cdot 11^{\gamma
_{1}}}{v }\right)\leq 11.
\end{equation}
Recall that a prime factor $q$ of $L_{m}$ is called \textit{primitive} if $%
q\nmid L_{k}$ holds for any $0<k<m$ and also $q\nmid (\eta -{\overline{\eta }%
})^{2}.$ When $q$ exists, it satisfies the congruence $q\equiv \pm 1\pmod m$
where the sign coincides with $(\frac{-4dv^2}{q})=(\frac{-d}{q}).$ Here, and
in what follows, $(\frac{a}{q})$ stands for the Legendre symbol of the
integer $a$ with respect to the odd prime $q$. Recall that a particular
instance of the Primitive Divisor Theorem for the Lucas sequences implies
that if $p\geq 5$, then $L_{p}$ always has a primitive prime factor except
for finitely many pairs $(\eta,{\overline{ \eta} })$ all of which appear in
Table 1 in \cite{Bilu} (see also \cite{1}). These exceptional Lucas numbers
are called \textit{defective}.

Let us first assume that we are dealing with a number $L_{p}$ without
primitive divisors. Then a quick look at Table 1 in \cite{Bilu} reveals that
the only defective Lucas numbers whose roots are in ${\mathbb{K}}={\mathbb{Q}%
}(i\sqrt{d})$ with $d$ appearing in the list \eqref{d} is $(\eta,~{\overline{%
\eta}})=((1+i\sqrt{11})/2,~(1-i\sqrt{11})/2)$ for which $L_5=1$ and $y=3$.
However, this is not convenient since for us $b>0$ and $x$ and $y$ are
coprime so $y$ cannot be a multiple of $3$.

Now let us look at the possibility when the Lucas number $L_{p}$ appearing
in the right hand side of equation \eqref{4.4} has a primitive divisor.
Since $p\geq 5$, it follows that $11$ is primitive for $L_{p}$. Thus, $%
11\equiv \pm 1\pmod p$. We now see that the only possibility is $p=5$ and
since $11\equiv 1\pmod 5$, we get that $(\frac{-d}{11})=1.$ Since $d\in
\{1,~2,~3,~6,~11,~22,~33,~66\}$, we get that $d=2$ and $d=6$. In particular, 
$u$ and $v$ integers.

In the remaining of this section, we shall treat each one of these two cases
separately.
\end{proof}

\subsection{The case $d=2$}

Since $P(L_{n})=11$ is coprime to $-4dv^{2}=-8v^2,$ we get the possibilities 
\begin{equation*}
v=\pm 2^{\alpha _{1}},\quad v=\pm 3^{\beta _{1}},\quad v=\pm 2^{\alpha
_{1}}3^{\beta _{1}}.
\end{equation*}
Since $y=u^{2}+2v^{2},$ we get that $u$ is odd.

\textbf{Case 1:} $v=\pm 2^{\alpha _{1}}.$

In this case, equation \eqref{4.4} becomes 
\begin{equation*}
\pm 3^{\beta _{1}}11^{\gamma _{1}}=5u^{4}-20u^{2}v^{2}+4v^{4}.
\end{equation*}
Since $u$ is odd, it follows that the right hand side of the last equation
above is congruent to $5\pmod 8$. So $\pm 3^{\beta _{1}}11^{\gamma
_{1}}\equiv 5\pmod 8$, showing that the sign on the left hand side is
negative and $\beta _{1}+\gamma _{1}$ is odd.

Assume first that $\beta _{1}=2\beta _{0}+1$ be odd. Then $\gamma
_{1}=2\gamma _{0}.$ We get 
\begin{equation*}
-3V^{2}=5U^{4}-20U^{2}+4,
\end{equation*}
where $(U,V):=({u}/{v},{3^{\beta _{0}}11^{\gamma _{0}}}/{v^{2}})$ is a $%
\{2\} $-integral point on the above elliptic curve. Using MAGMA we get no
solution.

Assume now that $\beta _{1}=2\beta_0$ is even. Then $\gamma _{1}=2\gamma_0+1$
is odd and we get that 
\begin{equation*}
-11V^{2}=5U^{4}-20U^{2}+4,
\end{equation*}
where $(U,V):=({u}/{v},{3^{\beta _{0}}11^{\gamma _{0}}}/{v^{2}})$ is a $%
\{2\} $-integral point on the above elliptic curve. With MAGMA, we get that
the only such points on the above curve are $(U,V)=(\pm 1,\pm 1)$ and $(\pm
1/2,\pm 1/4)$, leading to $(u,v)=(\pm 1,\pm 1)$ and $(\pm 1,\pm 2)$,
respectively. These give the solution $(1,3,1,0,2)$ and $(241,9,3,0,2)$
respectively. The second solution is also a solution for $n=10$ for which $%
y=3$. However, as $b=0$, we ignore such solutions.

\textbf{Case 2:} $v=\pm 3^{\beta _{1}}$.

In this case, the equation \eqref{4.4} becomes 
\begin{equation*}
\pm 2^{\alpha _{1}}11^{\gamma _{1}}=5u^{4}-20u^{2}v^{2}+4v^{4}.
\end{equation*}
Since $u$ is odd, the right hand side is congruent to $5\pmod 8.$ So we get the congruence $%
\pm 2^{\alpha _{1}}11^{\gamma _{1}}\equiv 5\pmod 8$. When $\alpha _{1}>0$,
there is no solution, while when $\alpha_{1}=0$, we have $-11^{\gamma
_{1}}\equiv 5\pmod 8$. Thus, $\gamma_{1}$ is odd and we get that 
\begin{equation*}
-11V^{2}=5U^{4}-20U^{2}+4,
\end{equation*}
where $(U,V):=({u}/{v},{11^{\gamma _{0}}}/{v^{2}})$ is a $\{2,3\}$-integral
point on the above elliptic curve. With MAGMA, we get a few points which
lead to two solutions of the original equation \eqref{eq:(1.2)} which are
not convenient for us since they have $b=0$.

\textbf{Case 3:} $v=\pm 2^{\alpha _{1}}3^{\beta _{1}}$.

In this case, the equation \eqref{4.4} becomes 
\begin{equation*}
\pm 11^{\gamma _{1}}=5u^{4}-20u^{2}v^{2}+4v^{4}.
\end{equation*}
Similar to the above cases, there is no solution with $b>0$.

\subsection{The case $d=6$}

In this case, for the equation \eqref{4.4} we get the possibilities 
\begin{equation*}
v=\pm 2^{\alpha _{1}},\quad v=\pm 3^{\beta _{1}},\quad v=\pm 2^{\alpha
_{1}}3^{\beta _{1}}.
\end{equation*}
Since $y=u^{2}+6v^{2}$, we get that $u$ is odd.

\textbf{Case 1:} $v=\pm 2^{\alpha _{1}}.$

Here, the equation \eqref{4.4} becomes 
\begin{equation*}
\pm 3^{\beta _{1}}11^{\gamma _{1}}=5u^{4}-60u^{2}v^{2}+36v^{4}.
\end{equation*}
Since $u$ and $v$ are coprime, we get 
\begin{equation*}
\pm 3^{\beta _{1}}11^{\gamma _{1}}\equiv 5\pmod 8.
\end{equation*}
So $\pm 3^{\beta _{1}}11^{\gamma _{1}}\equiv 5\pmod 8$, showing that the
sign on the left hand side is negative and that $\beta _{1}+\gamma _{1}$ is
odd.

First let $\beta _{1}=2\beta _{0}+1$ and $\gamma _{1}=2\gamma _{0}$. Then we
get that 
\begin{equation*}
-3V^{2}=5U^{4}-60U^{2}+36,
\end{equation*}
where $(U,V):=({u}/{v},{3^{\beta _{0}}11^{\gamma _{0}}}/{v^{2}})$ is a $%
\{2\} $-integral point on the above elliptic curve. Using MAGMA, we get no
solution.

Assume next that $\beta _{1}=2\beta_0$ be even. Then $\gamma
_{1}=2\gamma_0+1 $, and we have 
\begin{equation*}
-11V^{2}=5U^{4}-60U^{2}+36,
\end{equation*}
where $(U,V):=({u}/{v},{3^{\beta _{0}}11^{\gamma _{0}}}/{v^{2}})$ is a $%
\{2\} $-integral point on the above elliptic curve. With MAGMA, we obtain
that the only solutions as $(U,V)=(\pm 3,\pm 3)$ and $(\pm 9/4,\pm 57/16),$
leading to $(u,v)=(\pm 3,\pm 1)$. This leads to the solution $(837,15,1,5,2)$
of the initial equation \eqref{eq:(1.2)}, but as $\gcd (837,15)=3\neq 1$,
this is not a convenient solution.

\textbf{Case 2:} $v=\pm 3^{\beta _{1}}$.

In this case, the equation \eqref{4.4} becomes 
\begin{equation*}
\pm 2^{\alpha _{1}}11^{\gamma _{1}}=5u^{4}-60u^{2}v^{2}+36v^{4}.
\end{equation*}
This leads to 
\begin{equation*}
\pm 2^{\alpha _{1}}11^{\gamma _{1}}\equiv 5\pmod 8,
\end{equation*}
and when $\alpha _{1}>0$, we get no solution, while when $\alpha _{1}=0,$ we
get that 
\begin{equation*}
-11^{\gamma _{1}}\equiv 5\pmod 8.
\end{equation*}
Therefore $\gamma _{1}$ is odd, and we have 
\begin{equation*}
-11V^{2}=5U^{4}-60U^{2}+36,
\end{equation*}
where $(U,V):=({u}/{v},{2^{\alpha _{0}}11^{\gamma _{0}}}/{v^{2}})$ is a $%
\{3\}$-integral point on the above elliptic curve. With MAGMA we obtain the
only solutions as $(U,V)=(\pm 3,\pm 3)$ and $(\pm 9/4,\pm 57/16)$. They do
not lead to solutions of our original equation.

\textbf{Case 3: }$v=\pm 2^{\alpha _{1}}3^{\beta _{1}}$.

In this case, the equation \eqref{4.4} becomes 
\begin{equation*}
\pm 11^{\gamma _{1}}=5u^{4}-60u^{2}v^{2}+36v^{4}.
\end{equation*}
We now get that $\pm 11^{\gamma _{1}}\equiv 5\pmod 8$. This last congruence
is possible only when the left hand side has minus sign. Therefore $\gamma
_{1}$ must be odd. We are led to computing the $\{2,3\}-$integer points on a certain
elliptic curve and similarly as in the previous cases, we get no solution.


\begin{thebibliography}{99}
\bibitem{1} ABOUZAID, M.: \textit{Les nombres de Lucas et Lehmer sans
diviseur primitif}, J. Th. Nombres Bordeaux\/ \textbf{18} (2006), 299--313.

\bibitem{Arif1} ABU MURIEFAH, F.~S.$-$ARIF, S.~A.: \textit{The Diophantine
equation }$x^{2}+3^{m}=y^{n}$, Int. J. Math. Math. Sci.\textit{\/} \textbf{21%
} (1998), 619--620.

\bibitem{Arif2} ABU MURIEFAH, F.~S.$-$ARIF, S.~A.: \textit{The Diophantine
equation }$x^{2}+5^{2k+1}=y^{n}$, Indian J. Pure Appl. Math.\textit{\/} 
\textbf{30} (1999), 229--231.

\bibitem{Abu2001} ABU MURIEFAH, F.~S.$-$ARIF, S.~A.: \textit{On the
Diophantine equation }$x\sp2+2\sp k=y\sp n$. II, Arab J. Math. Sci.\textit{\/%
} \textbf{7} (2001), 67--71.

\bibitem{FSAbu7} ABU MURIEFAH, F.~S.$-$ARIF, S.~A.: \textit{On the
Diophantine equation }$x^{2}+5^{2k}=y^{n}$, Demonstratio Math.\textit{\/} 
\textbf{319} (2006), 285--289.

\bibitem{FSAbu3} ABU MURIEFAH, F.~S.$-$LUCA, F.$-$TOGB\'E, A.: \textit{On the
Diophantine equation }$x^{2}+5^{a}13^{b}=y^{n}$, Glasgow Math.\textit{\ J.\/}
\textbf{50} (2008), 175--181.

\bibitem{Berczes2} B\'ERCZES, A.$-$BRINDZA, B.$-$HAJDU, L.: \textit{On the
power values of polynomials}, Publ. Math. Debrecen\/, \textbf{53} (1998),
375--381.

\bibitem{Berczes} B\'ERCZES, A.$-$PINK, I.: \textit{On the Diophantine
equation }$x^{2}+q^{2k}=y^{n}$, Archiv der Math. (Basel)\textit{.\/} 
\textbf{91} (2008), 505--517.

\bibitem{Bilu} BILU, Y.~F.$-$HANROT, G.$-$VOUTIER, P.~M.: \textit{Existence
of primitive divisors of Lucas and Lehmer numbers. With an appendix by
M.Mignotte}, J.Reine Angew. Math.\textit{\ \/} \textbf{539} (2001), 75--122.

\bibitem{Bosma} BOSMA, W.$-$CANNON, J.$-$PLAYOUST, C.: \textit{The Magma
Algebra System I}. The user language, J. Symbolic Comput.\textit{\/} \textbf{%
24} (1997), 235--265.

\bibitem{Bugeaud2} BUGEAUD, Y.$-$MIGNOTTE, M.$-$SIKSEK, S.: \textit{%
Classical and modular approaches to exponantial Diophantine equations II.
The Lebesque- Nagell equation}, Compositio. Math.\/ \textbf{142} (2006),
31--62.

\bibitem{CDLPS} CANGUL, I.~N.$-$DEMIRCI, M.$-$LUCA, F.$-$PINTER, A.$-$%
SOYDAN, G.: \textit{On the Diophantine equation }$x^{2}+2^{a}11^{b}=y^{n}$,
Fibonacci Quart.\textit{\/}, \textbf{48}(2010), 1, 39--46.

\bibitem{Cohn2} COHN, J.~H.~E.: \textit{The Diophantine equation }$%
x^{2}+2^{k}=y^{n}$, Archive der Math. (Basel)\textit{.} \textbf{59} (1992),
341--344.

\bibitem{Cohn1} COHN, J.~H.~E.: \textit{The Diophantine equation }$%
x^{2}+C=y^{n}$, Acta Arith.\/ \textbf{65} (1993), 367--381.

\bibitem{Luca5} GOINS, E.$-$LUCA, F.$-$TOGB\'E, A.: \textit{On the Diophantine
equation }$x^{2}+2^{\alpha }5^{\beta }13^{\gamma }=y^{n}$, ANTS VIII
Proceedings: A. J. van der Poorten and A. Stein (eds.)\textit{\/}, ANTS
VIII, Lecture Notes in Computer Science \textbf{5011} (2008), 430--442.

\bibitem{Gyry} GY\"{O}RY, K.$-$PINK, I.$-$PINTER, A.: \textit{Power values of
polynomials and binomial Thue-Mahler equations}, Publ. Math. Debrecen\/ 
\textbf{65} (2004), 341--362.

\bibitem{Landau} LANDAU, E.$-$OSTROWSKI, A.: \textit{On the Diophantine
equation }$ay^{2}+by+c=dx^{n}$, Proc. London Math. Soc.\/ \textbf{19}
(1920), 276--280.

\bibitem{Le1} LE, M.~H.: \textit{On Cohn's conjecture concerning the
Diophantine equation }$x^{2}+2^{m}=y^{n}$, Archive der Math. (Basel).\textit{%
\/} \textbf{78} (2002), 26--35.

\bibitem{Lebesque} LEBESGUE, V.~A.: \textit{Sur I'impossibilit\'{e} en
nombres entieres de I'l equation }$x^{m}=y^{2}+1$, Nouvelles Ann. des. Math.%
\textit{\/} \textbf{9} (1850), 178--181.

\bibitem{Liqun} LIQUN, T.: \textit{On the Diophantine equation }$%
x^{2}+3^{m}=y^{n}$, INTEGERS\textit{\/} \textbf{8} (2008), A55.

\bibitem{Luca3} LUCA, F.:\textit{\ On a Diophantine equation}, Bull. Aus.
Math. Soc.\textit{\/} \textbf{61} (2000), 241--246.

\bibitem{Luca1} LUCA, F.: \textit{On the equation }$x^{2}+2^{a}3^{b}=y^{n}$,
Int. J. Math. and Math. Sci.\/ \textbf{29} (2002), 239--244.

\bibitem{Luca4} LUCA, F.$-$TOGB\'E, A.:\textit{\ On the Diophantine equation }$%
x^{2}+7^{2k}=y^{n}$, Fibonacci Quart.\textit{\/} \textbf{45} (2007),
322--326.

\bibitem{Luca2} LUCA, F.$-$TOGB\'E, A.: \textit{On the Diophantine equation }$%
x^{2}+2^{a}5^{b}=y^{n}$, Int. J. Number Theory\textit{\/} \textbf{4} (2008),
973--979.

\bibitem{Mignotte} MIGNOTTE, M.$-$DE WEGER, B.~M.~M.: \textit{On the
Diophantine equations }$x^{2}+74=y^{5}$\textit{\ and }$x^{2}+86=y^{5}$,
Glasgow Math. J.\textit{\/} \textbf{38} (1996), 77--85.

\bibitem{Nagell} NAGELL, T.: C\textit{ontributions to the theory of a
category of diophantine equations of the second degree with two unknowns},
Nova Acta Reg. Soc. Upsal. Ser.\/ \textbf{4 } (1955), 1--38.

\bibitem{Pink} PINK, I.: \textit{On the Diophantine equation }$%
x^{2}+2^{a}3^{b}5^{c}7^{d}=y^{n}$, Publ. Math. Debrecen\textit{\ \/} \textbf{%
70} (2007), 149--166.

\bibitem{SS} SHOREY, T.~N.$-$STEWART, C.~L.: \textit{On the Diophantine
equation }$ax\sp{2t}+bx\sp{t}y+cy\sp{2}=d$\textit{\ and pure powers in
recurrence sequences}, Math. Scand.\/ \textbf{52} (1983), 24--36.

\bibitem{Tengely} TENGELY, S.: \textit{On the Diophantine equation }$%
x^{2}+a^{2}=2y^{p}$, Indag. Math. (N.S.)\/ \textbf{15} (2004), 291--304.
\end{thebibliography}
\end{document}